\documentclass[a4paper,12pt]{article}
\usepackage[left=2.5cm,right=2.5cm,top=2.5cm,bottom=2.5cm]{geometry} 

\usepackage[T1]{fontenc}
\usepackage{times}

\usepackage{amsmath,amssymb,amsthm,hyperref}

\usepackage{bm}

\usepackage{lineno}

\usepackage[vlined, ruled, lines numbered]{algorithm2e}

\usepackage{tikz}
\usetikzlibrary{trees,shapes,arrows}

\usepackage{listings}

\newcommand{\bes}{\begin{equation*}}
\newcommand{\ees}{\end{equation*}}
\newcommand{\be}{\begin{equation}}
\newcommand{\ee}{\end{equation}}
\newcommand{\bbm}{\begin{bmatrix}}
\newcommand{\ebm}{\end{bmatrix}}

\newcommand{\R}{\mathbb{R}}
\newcommand{\Rn}{\R^n}
\newcommand{\Rnn}{\R^{n \times n}}
\newcommand{\Rm}{\R^m}
\newcommand{\Rmm}{\R^{m \times m}}
\newcommand{\C}{\mathbb{C}}

\newcommand{\Cm}{\C^m}

\newcommand{\N}{\mathbb{N}}

\newtheorem{theorem}{Theorem}[section]
\newtheorem{lemma}[theorem]{Lemma}

\newtheorem{remark}[theorem]{Remark}

\usepackage{authblk,url}

\author{Carl Christian Kjelgaard Mikkelsen \\ \texttt{spock@cs.umu.se}} \affil{Department of Computing Science and HPC2N \\ Ume{\aa} University \\ Sweden}
\title{Well-conditioned eigenvalue problems that  overflow}

\begin{document}

\maketitle

\abstract{In this note we present a parameterized class of lower triangular matrices.
The components of the eigenvectors grow rapidly and will exceed the representational range of any finite number system.
The eigenvalues and the eigenvectors are well-conditioned with respect to componentwise relative perturbations of the matrix. 
This class of matrices is well suited for testing software for computing eigenvectors as these routines must be able to handle overflow successfully. }

\section{Introduction} \label{sec:introduction}

Given a matrix$\bm{A} \in \Rmm$ the standard eigenvalue problem consists of finding eigenvalues $\lambda \in \C$ and eigenvectors $\bm{x} \in \Cm$ such that
\bes
\bm{A} \bm{x} = \lambda \bm{x}.
\ees
If $\bm{A}$ is a dense nonsymmetric matrix then the standard eigenvalue problem is often solved by first reducing $\bm{A}$ to Hessenberg form 
\bes
\bm{H} = \bm{Q}_1^T \bm{A} \bm{Q}_1
\ees
and then to real Schur form 
\bes
\bm{S} = \bm{Q}_2^T \bm{H} \bm{Q}_2
\ees
using orthogonal similarity transformations.
Here the matrix $\bm{H}$ is upper Hessenberg and the matrix $\bm{S}$ is upper quasi-triangular with diagonal blocks that are either $1$-by-$1$ or $2$-by-$2$.
The eigenvalues of $\bm{A}$ can be determined from the diagonal blocks of $\bm{S}$.
Every 1-by-1 diagonal block on the diagonal of $\bm{S}$ is a real eigenvalue of $\bm{A}$ and every $2$-by-$2$ block on the diagonal of $\bm{S}$ specifies a pair of complex conjugate eigenvalues of $\bm{A}$.
The eigenvectors of $\bm{S}$ can be computed using a variant of substitution and transformed to eigenvectors of $\bm{A}$.
Specifically, if $\bm{y} \not = 0$ satisfies $\bm{S}\bm{y} = \lambda \bm{y}$, then $\bm{x} = \bm{Q}_1\bm{Q}_2 \bm{y}$ is an eigenvector of $\bm{A}$ corresponding to the eigenvalue $\lambda$.

However, substitution is very vulnerable to floating point overflow and special software has been developed to handle this problem.
We say that an algorithm is robust if all intermediate and final results are in the representable range, i.e., overflow is prevented.
In LAPACK \cite{lapack} the robust subroutines for computing eigenvectors are all derived from {\tt xlatrs}  \cite{anderson1991robust-triangular}.
In Elemental \cite{moon2016accelerating-eigenvector} there are parallel subroutines that can be used to compute eigenvectors of triangular matrices, but they are not fully robust.
In StarNEig \cite{myllykoski2020introduction-to} there are parallel robust subroutines for computing standard and generalized eigenvectors from matrices or matrix pairs in real Schur form.

The contribution of this note is to exhibit a class of lower triangular matrices that can be used to test robust solvers. The components of the eigenvectors grow rapidly and will exceed the representational range of any finite set of numbers sooner rather than later. 
Moreover, the eigenvalues and the eigenvectors are well-conditioned with respect to componentwise relative  perturbations of the matrix.
These results are all established in Section \ref{sec:main-results}.
We briefly cover the transformation to upper triangular problems in Section \ref{sec:upper-triangular-problems} and finish with some concluding remarks in Section \ref{sec:conclusion}.

Here we offer the following small example as an appetizer. Consider the matrices $\bm{A}$ and $\bm{X}$ given by 
\bes
\bm{A} = 
\bbm 
1 & & & & \\
-5 & 2 & & & \\
-5 & -5 & 3 & & \\
-5 & -5 & -5 & 4 & \\
-5 & -5 & -5 & -5 & 5 
\ebm, \quad \bm{X} = \bbm 
1 & & & & \\
5 & 1 & & & \\
15 & 5 & 1 & & \\
35 & 15 & 5 & 1 & \\
70 & 35 & 15 & 5 & 1
\ebm.
\ees
It is straightforward to verify that the $j$th column of $\bm{X}$ is an eigenvector of $\bm{A}$ corresponding to the eigenvalue $\lambda_j = j$ and $x_{ij} \ge 2^{i-j}$. 
However, by Theorem \ref{thm:eigenvectors} and Lemma \ref{lem:exponential-growth} these properties are preserved when our example is generalized to any dimension $m$.
Moreover, the nontrivial components of the  first column of $\bm{X} = [x_{ij}]$ can be obtained by solving the linear system
\bes
\bbm
1 &  & & \\
-5 & 2 & & \\
-5 & -5 & 3 & \\
-5 & -5 & -5 & 4 
\ebm
\bbm
x_{21} \\ 
x_{31} \\ 
x_{41} \\
x_{51}
\ebm
=
\bbm
5 \\ 5 \\ 5 \\ 5
\ebm.
\ees
In this case, forward substitution consists of adding and dividing real numbers that are strictly positive. Hence it is not surprising that this eigenvector is  well-conditioned with respect to componentwise relative perturbations of the matrix $\bm{A}$. 
A general upper bound for the relevant condition numbers is established as Theorem \ref{thm:skeel-condition-number:bound}.

\section{Auxiliary results} \label{sec:auxiliary-results}
In this section we derive a set of elementary results related to the solution of very special triangular linear systems. 
These results are used to prove the main results.

Consider the very special linear system given by
\be \label{equ:general:system}
\bm{G} \bm{x} :=
\bbm
d_1 & & & \\
-c & d_2 & & \\
\vdots & \ddots & \ddots & \\
-c & \dots & -c & d_m
\ebm  \bbm x_1 \\ x_2 \\ \vdots \\ x_m \ebm = \bbm c \\ c \\ \vdots \\ c \ebm =: \bm{f}
\ee
where $\bm{G} \in \Rmm$ and $\bm{f} \in \Rm$.  
If the diagonal elements are nonzero, then $\bm{G}$ is nonsingular and the unique solution $\bm{x}$ can be found using forward substitution. The familiar formula takes the form:
\be \label{equ:substitution}
x_k = a_k \left(1 + \sum_{j=1}^{k-1} x_j\right), \quad a_k = \frac{c}{d_k}.
\ee
The components of $\bm{x}$ can also be expressed compactly using the following theorem.
\begin{theorem} \label{thm:solution} 
If $d_j \not =0$ for all $j$, then the solution of the linear system \eqref{equ:general:system} is given by
  \bes
  x_k = a_k \omega_k
  \ees
  where
  \bes
  \forall k \in \{1,2,\dotsc,m\} : \: \omega_k = \prod_{j=1}^{k-1} ( 1+a_j), \quad a_k = \frac{c}{d_k}.
  \ees
\end{theorem}

\begin{proof} The proof is  split into two steps.
\begin{enumerate}
\item We begin by showing that 
\be \label{equ:product}
\omega_k = 1 + \sum_{i=1}^{k-1} a_i \omega_i.
\ee
Let $S_m = \{1,2,\dotsc,m\}$ and let $V \subseteq S_m$ be given by
  \bes
  V = \left \{ j \in S_m \: : \: \omega_j = 1 + \sum_{i=1}^{j-1} a_i \omega_i \right \}.
  \ees
  It is clear that $1 \in V$, because $\omega_1 = 1$. Now assume that $S_m \setminus V \not = \emptyset$. Then $S_m \setminus V$ has a smallest element $k$. We must have $k > 1$ because $1 \in V$. By definition of $\omega_k$ we have
  \bes
  \omega_k = (1 + a_{k-1})  \omega_{k-1} = \omega_{k-1} + a_{k-1} \omega_{k-1}.
  \ees
  Since $k$ is the smallest element of $V$ we must have $k-1 \in V$ and we can therefore write
  \bes
  \omega_{k-1} = 1 + \sum_{j=1}^{k-2} a_i \omega_i.
  \ees
  It follows that
  \bes
    \omega_k = \omega_{k-1} + a_{k-1} \omega_{k-1} = \left [ 1 + \sum_{i=1}^{k-2} a_i \omega_i \right] + a_{k-1} \omega_{k-1} \\
    = 1 + \sum_{i=1}^{k-1}  a_i \omega_i.
  \ees
  We conclude that $k \in V$. This is a contradiction, because $k \in S_m \setminus V$. Therefore, we must have $V = S_m$.

\item We will now show that
\bes
    x_k = a_k \omega_k. 
\ees
Let $W \subseteq S_m = \{1,2,\dotsc,m\}$ be given by
  \bes
  W = \{ j \in S_m \: : \: x_j = a_j \omega_j \}.
  \ees
  It is clear that $1 \in W$ because $\omega_1 = 1$ and $x_1 = a_1$ implies $x_1 = a_1 \omega_1$. Now assume that $S_m \setminus W \not = \emptyset$. Then $S_m \setminus W$ has a smallest element $k$. Equation \eqref{equ:substitution} is simply the statement that
  \bes
  x_k = a_k \left(1 + \sum_{j=1}^{k-1} x_j \right).
  \ees
  Now since $k$ is the smallest element of $S_m \setminus W$ we have $j \in W$ for $j < k$ or equivalently $x_j = a_j \omega_j$ for $j < k$. It follows that
  \bes
  x_k = a_k \left(1 + \sum_{j=1}^{k-1} a_j \omega_j\right).
    \ees
    Equation \eqref{equ:product} now implies that
    \bes
    x_k = a_k \omega_k
    \ees
    and $k \in W$. This is a contradiction because $k \in S_m\setminus W$.
    We conclude that $W = S_m$.
    \end{enumerate}
    This completes the proof.
\end{proof}

We will now state a formula for the inverse matrix $\bm{H} = \bm{G}^{-1}$. 
\begin{theorem} \label{thm:inverse} If $d_j \not = 0$ for all $j$, then $\bm{G}$ is nonsingular and the components of the inverse matrix $\bm{G}^{-1} = \bm{H} = [h_{ij}]$ are given by
  \bes
  h_{ij} =
  \begin{cases} 
    0 & i < j, \\[10pt]
    \displaystyle 
    \frac{1}{d_j} & i = j, \\[10pt]
    \displaystyle 
    \frac{a_i}{d_j}  \frac{\omega_i}{\omega_{j+1}} & i > j.
  \end{cases}
  \ees
\end{theorem}

\begin{proof} Let $\bm{x} = \bbm x_1 & x_2 & \dotsc & x_m \ebm^T \in \Rm$ denote the $j$th column of the inverse matrix $\bm{H}$. Since $\bm{G}$ is lower triangular, we have 
\bes
x_i = \begin{cases} 
0 & i < j, \\[10pt]
\displaystyle
\frac{1}{d_j} & i = j.
\end{cases}
\ees
The remaining components of $\bm{x}$ are obtained by solving the linear system
  \bes
  \bbm
  d_{j+1} & & & \\
  -c & d_{j+2} & & \\
  \vdots & \ddots & \ddots & \\
  -c & \dots & -c & d_m
  \ebm
  \bbm
  x_{j+1} \\
  x_{j+2} \\
  \vdots \\
  x_m
  \ebm =
  x_j \bbm c \\ c \\ \vdots \\ c \ebm.
  \ees
  Now let $i \in \{1,2,\dots,m-j\}$. Then by Theorem \ref{thm:solution} we have
  \bes
  x_{j+i} = x_j \left( a_{j+i} \prod_{k=1}^{i-1} (1 + a_{j+k}) \right) = x_j a_{j+i} \prod_{k=j+1}^{j+i-1} (1 + a_k) = x_j a_{j+i} \frac{\omega_{j+i}}{\omega_{j+1}}
  \ees
  This shows that
  \bes
  x_i = \frac{a_i}{d_j} \frac{\omega_i}{\omega_{j+1}}, \quad i > j.
  \ees
  This completes the proof.
  \end{proof}

  Skeel's condition number for a nonsingular linear system $\bm{A} \bm{x} = \bm{b}$ where $\bm{A} \in \Rmm$ and $\bm{b} \in \Rm$ depends explicitly on the solution $\bm{x}$ and is given by \bes
  \kappa_\infty(\bm{A},\bm{x}) = \frac{\| |\bm{A}^{-1}||\bm{A}||\bm{x}| \|_\infty}{\|\bm{x}\|_\infty}.
  \ees
  The following result will be used to compute Skeel's condition number for systems of the form given by equation \eqref{equ:general:system}.
  
 \begin{lemma} \label{lem:skeel:intermediate-results} Assume $d_j > 0$ for all $j$ and $c>0$. Let $\bm{x}$ denote the solution of the linear system \eqref{equ:general:system}. Let $\bm{y} = |\bm{G}||\bm{x}|$ and let $\bm{z} = |\bm{G}^{-1}|\bm{y}$. Then
    \bes
    y_i = c (2 \omega_i - 1)
    \ees
    and
    \bes
    z_i = a_i (2 \omega_i - 1) + \sum_{j=1}^{i-1} a_i a_j \frac{\omega_i}{\omega_{j+1}} (2 \omega_j - 1).
    \ees
  \end{lemma}
  \begin{proof}
    By Theorem \ref{thm:solution} the solution of $\bm{G}\bm{x} = \bm{f}$ is given by
    \bes
    x_k = a_k \omega_k
    \ees
    where our assumptions ensure that
    \bes
    a_k = \frac{c}{d_k} > 0, \quad \omega_k = \prod_{j=1}^{k-1} (1 + a_j) > 0.
    \ees
    This shows that $x_k > 0$ for all $k$, so $|\bm{x}| = \bm{x}$. 
    Since $|\bm{G}|$ is lower triangular, we have
    \begin{multline*}
    y_i = (|\bm{A}||\bm{x}|)_i = d_i x_i + \sum_{j=1}^{i-1} c x_j 
    = d_i a_i \omega_i + c \sum_{j=1}^{i-1} a_i \omega_i \\
    = c \omega_i + c \left[ \omega_i - 1\right] 
    = c ( 2\omega_i - 1).
    \end{multline*}
    Here we have used equation \eqref{equ:product} during the last reduction.
    By Theorem \ref{thm:inverse} the elements of $\bm{H} = \bm{G}^{-1}$ are nonnegative. Hence $|\bm{G}^{-1}| = \bm{G}^{-1}$. Since $\bm{H} = \bm{G}^{-1}$ is lower triangular, we have
    \begin{multline*}
    z_i = (|\bm{G}^{-1}| \bm{y})_i = \left(  \bm{H}  \bm{y} \right)_i = h_{ii} y_i + \sum_{j=1}^{i-1} h_{ij} y_j \\  = a_i (2 \omega_i - 1) + \sum_{j=1}^{i-1} a_i a_j \frac{\omega_i}{\omega_{j+1}} ( 2 \omega_j - 1). 
  \end{multline*}
  This completes the proof.
  \end{proof}

\section{Main results} \label{sec:main-results}
In this section we present a class of lower triangular matrices parameterized using three real numbers $a$, $b$, and $c$. 
The components of the eigenvectors tend to infinity when $\gamma = \frac{c}{b} > 1$ and the growth is at least exponential when $\gamma \ge m$ where $m$ is the dimension of the matrix. 
Regardless, the eigenvectors are well-conditioned with respect to componentwise relative perturbations of the matrix $\bm{A}$ when $b>0$ and $c>0$.

Let $a, b, c \in \R$ and consider the lower triangular matrix $A \in \Rmm$ given by
\be \label{equ:A}
a_{ij} = \begin{cases}
  0 & i < j, \\
  a+jb & i = j, \\
  -c & i > j.
\end{cases} 
\ee
The case of $m=4$ is illustrated by the matrix $\bm{A}$ given by 
\bes
\bm{A} = \bbm
a + b& & &  \\
-c & a+2b & & \\
-c & -c & a+3b & \\
-c & -c & -c & a+4b \\
\ebm.
\ees
The eigenvalues of $\bm{A}$ can be read off from the diagonal of $\bm{A}$, i.e.,
\bes
\lambda_j = a + jb
\ees
and they are trivially well-conditioned with respect to componentwise relative perturbations of the matrix.
If $b \not = 0$, then the eigenvalues are distinct and $\bm{A}$ is diagonalizable.
The eigenvectors are determined up to a scaling by the sequence $\{z_k\}_{k=0}^\infty$ given by
\be \label{equ:sequence}
z_k = \binom{\gamma + k - 1}{k}, \quad \gamma = \frac{c}{b}.
\ee
Here $\binom{x}{k}$ denotes the binomial coefficient given by
\bes
  \forall x \in \R \: \: \forall k \in \N_0\: :\:   \binom{x}{k} = \frac{\prod_{i=0}^{k-1}(x-i)}{k!} = \frac{x(x-1)(x-2) \dotsc (x-k+1)}{k!}.
\ees
Specifically, we have the following theorem.
\begin{theorem} \label{thm:eigenvectors}
  Let $\bm{A} \in \Rmm$ be given by equation \eqref{equ:A} where $a, b, c \in \R$ and $b \not = 0$. Let $\bm{X} \in \Rmm$ denote the lower triangular matrix given by
  \bes
  x_{ij} =
  \begin{cases}
    0 & i < j, \\
    z_{i-j} & i \ge j.
  \end{cases}
  \ees
  Then the $j$th column of $\bm{X}$ is an eigenvector of $\bm{A}$ with respect to the eigenvalue $\lambda_j = a + jb$.
\end{theorem}

\begin{proof} Consider the problem of computing an eigenvector $\bm{x}$ of $\bm{A}$ with respect to the eigenvalue $\lambda_l = a + lb$. We have $x_i = 0$ for $i<l$ and we are free to choose $x_l= z_0 = 1$, provided the nontrivial components of $\bm{x}$, i.e., the vector $\bbm x_{l+1},\dotsc,x_m\ebm^T \in \R^{m-l}$ solves the linear system
  \bes
  \bbm
  b & & & \\
  -c & 2b & & \\
  \vdots & \ddots & \ddots & \\
  -c & \dots & -c & (m-l)b \\
  \ebm
  \bbm
  x_{l+1} \\ x_{l+2} \\ \vdots \\ x_{m}
  \ebm =
  \bbm
  c \\ c \\ \vdots \\ c 
  \ebm.
  \ees
  This linear system has dimension $m-l$. It is a special case of Theorem \ref{thm:solution} and  corresponds to the choice of $d_j = jb$ and $a_j = \frac{\gamma}{j}$. Let $y_i = x_{l+i}$ for $i=1,2\dotsc,m-l$, then by Theorem \ref{thm:solution} we have
  \bes
  y_k = a_k \omega_k = a_k \prod_{j=1}^{k-1} ( 1 + a_j) = 
  \frac{\gamma}{k} \prod_{j=1}^{k-1} \left(1 + \frac{\gamma}{j} \right) = \frac{\gamma}{k} \prod_{j=1}^{k-1} \frac{\gamma+j}{j} = \binom{\gamma+k-1}{k} = z_k.
  \ees
  We conclude that 
  \bes
  x_i =
  \begin{cases}
    0 & i < l, \\
    z_{i-l} & i \ge l.
  \end{cases}
  \ees
  This completes the proof.
\end{proof}

The behavior of the sequence given by equation \eqref{equ:sequence} is entirely controlled by the ratio $\gamma = \frac{c}{b}$. In particular, we have the following lemma.

\begin{lemma} \label{lem:asymptotics} Let $\alpha \in \R$ be any real number and let $\{y_k\}_{k=0}^\infty \subset \R$ be the sequence given by
\bes
y_k = \binom{\alpha + k}{k}.
\ees
Then the following statements hold:
\begin{enumerate}
    \item 
    If $\alpha > 0$, then $\{y_k\}_{k=0}^\infty$ is increasing and
\bes
y_k \rightarrow \infty, \quad k \rightarrow \infty, \quad k \in \N_0.
\ees
\item If $\alpha = 0$, then $y_k = 1$ for all $k \in \N_0$. 
\item If $\alpha < 0$ is an integer, then $y_k = 0$ for all sufficiently large $k \in \N_0$.

\item If $\alpha < 0$ is not an integer, then $y_k \not =0$, but
\bes
y_k \rightarrow 0, \quad k \rightarrow \infty, \quad k \in \N_0.
\ees
Moreover, the convengence is strictly monotone for all sufficently large $k$ and the rate of convergence is sublinear.
\end{enumerate}
\end{lemma}

\begin{proof} When studying the different cases it is convenient to exploit that $y_k$ can be rewritten as
\begin{equation} \label{equ:rewrite:yk}
    y_k = \frac{\prod_{j=1}^k(\alpha+j)}{k!} = \prod_{j=1}^k\left(1 + \frac{\alpha}{j} \right).
\end{equation}
\begin{enumerate}
    \item If $\alpha > 0$ then equation \eqref{equ:rewrite:yk} shows that $y_k$ is strictly positive and the sequence is increasing because
    \bes
        y_{k+1} = \left(1 + \frac{\alpha}{k+1} \right) y_k > y_k.
    \ees
    Now choose any $l \in \N$ such that $\alpha \leq l$ and consider any $k \ge l$. Then
    \bes 
        y_k = \prod_{j=1}^k \left(1+\frac{\alpha}{j} \right) = C p_k
    \ees
    where we have introduced
\begin{equation} \label{equ:factors}
    C = \prod_{j=1}^{l-1} \left(1+\frac{\alpha}{j} \right), \quad p_k = \prod_{j=l}^k \left(1+\frac{\alpha}{j} \right).
    \end{equation}
    By design, $C > 0$ is a constant that is independent of $k$.
    We now investigate the convergence of the sequence $\{p_k\}_{k=l}^{\infty}$. 
    We have
    \bes
    \log p_k = \sum_{j=l}^k \log\left(1 + \frac{\alpha}{j}\right) \ge \sum_{j=l}^k \frac{1}{2} \frac{\alpha}{j}
    \ees
    simply because 
    \bes
        \forall x \in [0,1] \: : \: \log(1+x) \ge \frac{1}{2}{x}.
    \ees
    Since the harmonic series is divergent we first conclude that 
    \bes
    p_k \rightarrow \infty, \quad k \rightarrow \infty, \quad k \ge l.
    \ees
    and then
    \bes
    y_k = Cp_k \rightarrow \infty, \quad k \rightarrow \infty, \quad k \in \N_0.
    \ees
    because $C>0$.
  \item If $\alpha = 0$, then for all $k \in \N_0$:
  \bes
      y_k = \frac{\prod_{j=1}^{k} (\alpha + j)}{k!} =\frac{\prod_{j=1}^{k} j}{k!} = \frac{k!}{k!} = 1.
  \ees
  \item If $\alpha < 0$ is an integer, then $\alpha = -l$ for exactly one $l \in \N$, and
  \bes
     \forall k \ge l \: : \: y_k = \frac{\prod_{j=1}^{k} (\alpha + j)}{k!} =  \frac{\prod_{j=1}^{k} (j-l)}{k!} = 0
  \ees
   simply because the term corresponding to $j=l$ is zero.
  \item If $\alpha < 0$ is not an integer, then equation \eqref{equ:rewrite:yk} shows that $y_k \not = 0$ because no term is zero. We also have $-l < \alpha < 1-l$ for exactly one $l \in \N$. Now let $k \ge l$, then we again write
  \bes
  y_k =  C p_k
  \ees
  where $C$ and $p_k$ are given by equation \eqref{equ:factors}. We again investigate the convergence of the sequence $\{p_k\}_{k=l}^{\infty}$. Since $\alpha < 0$ we have 
  \be \label{equ:pk:1}
      p_k = \prod_{j=l}^k \left(1-\frac{|\alpha|}{j} \right).
  \ee
  This shows that $p_k$ is a product of strictly positive terms because $|\alpha| < l$.
  It follows that 
  \bes
      - \log p_k = \sum_{j=l}^k -\log\left( 1 - \frac{|\alpha|}{j}\right) \ge \sum_{j=l}^k \frac{1}{2} \frac{|\alpha|}{j}
    \ees
    simply because
    \bes
        \forall x \in [0,1] \: : \: -\log(1-x) \ge -\frac{1}{2} x.
    \ees
    The divergence of the harmonic series now implies that
  \bes
  - \log(p_k) \rightarrow \infty, \quad k \rightarrow \infty, \quad k \ge l.
  \ees
  We can now conclude that 
  \bes
  p_k \rightarrow 0, \quad k \rightarrow \infty, \quad k \geq l.
  \ees
  It follows that
  \bes
  y_k = C p_k \rightarrow 0, \quad k \rightarrow \infty, \quad k \in \N_0.
  \ees
  The convergence is strictly monotone for $k \ge l$ simply because 
  \be
  p_{k+1} = p_k \left( 1 - \frac{|\alpha|}{k+1}\right) < p_k.
  \ee
  The convergence is sublinear, because
  \bes
  \frac{y_{k+1}}{y_k} = \frac{\alpha + k + 1}{k+1} \rightarrow 1, \quad k \rightarrow \infty, \quad k \in \N_0.
  \ees
  \end{enumerate}
  This completes the proof.
\end{proof}

By Theorem \ref{thm:eigenvectors} and Lemma \ref{lem:asymptotics} the eigenvectors of the matrix $\bm{A}$ will \emph{eventually} exceed the representational range provided $\gamma = \frac{c}{b} > 1$ and the dimension of the matrix is sufficiently large. 
However, if we increase the size of $\gamma$, then the eigenvectors grow at least exponentially. We have the following lemma.

\begin{lemma} \label{lem:exponential-growth} Let $\bm{X} = [x_{ij}] \in \Rmm$ denote the lower triangular matrix given by Theorem \ref{thm:eigenvectors} and let $\gamma$ denote the ratio $\gamma = \frac{c}{b}$. If $\gamma \ge m$, then
  \bes
  x_{ij} \ge 2^{i-j}, \quad 1 \leq j \leq i \leq m.
  \ees
\end{lemma}

\begin{proof} Let $i \ge j$ be given and set $k = i-j \ge 0$. Then
  \bes
  x_{ij} = z_k = \binom{\gamma+k-1}{k} = \prod_{l=0}^{k-1} \frac{\gamma+l}{l+1} \ge \prod_{l=0}^{k-1} \frac{m+l}{l+1}. 
  \ees
  We now claim that 
  \be \label{equ:quotient}
   \forall l \in \{0,1,2\dotsc, k  - 1\} \: : \: \frac{m + l}{l+1} \ge 2.
  \ee
  We have
  \bes
   \frac{m + l}{l+1} \ge 2  \quad \Leftrightarrow \quad m+l  \ge 2(l+1) \quad \Leftrightarrow \quad m \ge l + 2.
  \ees
  We have $l \leq k - 1$ and $k \leq m-1$. It follows that
  \bes
  l + 2 \leq (k-1) + 2 = k + 1 \leq (m-1) + 1 = m.
  \ees
  This shows that inequality \eqref{equ:quotient} is satisfied. The immediate implication is that
  \bes
  x_{ij} \geq 2^k = 2^{i-j}.
  \ees
  This completes the proof.
\end{proof}

\begin{remark} We emphasize that the growth of the components of the matrix $\bm{X}$ is independent of the location and clustering of the eigenvalues of $\bm{A}$. If $a \not = 0$, and if $\frac{mb}{a}$ is small, then the eigenvalues $\lambda_j = a + jb$ are clustered near $a$. If $a=0$ and $b\not=0$, then the eigenvalues are not clustered anywhere. In any case, it is the fraction $\gamma = \frac{c}{b}$ and not the value of $a$ that decides the behavior of the eigenvectors.
\end{remark}


We now study the conditioning of the eigenvectors of the matrix $\bm{A}$ given by equation \eqref{equ:A}. We consider perturbations $|\bm{A}|$ that are bounded componentwise relative to $\bm{A}$, i.e.,
\bes
|\Delta \bm{A}| \leq \epsilon |\bm{A}|,
\ees
where $\epsilon >0$ is a small number.
The eigenvectors of the matrix $\bm{A}$ are determined by the columns of the matrix $\bm{X}$ given in Theorem \ref{thm:eigenvectors}. The nontrivial components of the $j$th column of $\bm{X} = [x_{ij}]$ are computed as the solution of a linear system of the form
\be \label{equ:special:system}
\bm{B} \bm{x} =
\bbm
b & & &  \\
-c & 2b & & \\
\vdots & \ddots & \ddots & \\
-c & \dots & -c & n b 
\ebm \bbm
x_{j+1,j} \\ x_{j+2,j} \\ \vdots \\ x_{m,j}
\ebm=
\bbm
c \\ c \\ \vdots \\ c
\ebm
=: \bm{f}
\ee
where the dimension of the system is $n = m - j$. It is straightforward to verify that componentwise relative perturbations of $\bm{A}$ induces componentwise relative perturbations of $\bm{B}$ and $\bm{f}$. 
In this situation, the relevant condition number is Skeel's condition number given by
\bes
\kappa_\infty(\bm{B},\bm{f}) = \frac{\||\bm{B}^{-1}| |\bm{B}| |\bm{x}| \|_\infty}{\|\bm{x}\|_\infty}.
\ees
We have the following theorem.

\begin{theorem} \label{thm:skeel-condition-number:bound} Let $b > 0$ and $c>0$ satisfy $\gamma = \frac{c}{b} > 1$. Let $\bm{B} \in \Rnn$ and $\bm{f} \in \Rn$ be as in equation \eqref{equ:special:system}. Then Skeel's condition number satisfies
  \bes
  \kappa_\infty(\bm{B},\bm{f}) \leq 2 \left( 1 + \gamma \log \left(\frac{\gamma + n - 1}{\gamma} \right) \right).
  \ees
\end{theorem}

\begin{proof} Let $\bm{x} \in \Rn$ denote the solution of equation \eqref{equ:special:system}. Let $\bm{z} = |\bm{B}^{-1}||\bm{B}||\bm{x}|$. Our objective is to show that 
\bes
\frac{\|\bm{z}\|_\infty}{\|\bm{x}\|_\infty} \leq 2 \left( 1 + \gamma \log \left(\frac{\gamma + n - 1}{\gamma} \right) \right).
\ees 
By the definition of $\bm{z}$ we have $0 \leq z_i$. We will now bound $z_i$ from above. By Lemma \ref{lem:skeel:intermediate-results} we have
  \bes
  z_i = a_i ( 2 \omega_i - 1) + \sum_{j=1}^{i-1} a_i a_j \frac{\omega_i}{\omega_{j+1}}( 2 \omega_j-1)
    \ees
    where $a_i = \frac{c}{d_i} > 0$ and $\omega_{i} = \prod_{k=1}^{i-1} (1 + a_k) > 0$. It is clear that
    \bes
   z_i \leq 2 \left( a_i \omega_i + \sum_{j=1}^{i-1} a_i a_j \frac{\omega_i}{\omega_{j+1}} \omega_j \right) = 2 x_i \left(1  + \sum_{j=1}^{i-1} \frac{a_j}{1 + a_j} \right).
   \ees
   We now utilize that $a_j = \frac{c}{d_j} = \frac{c}{jb} = \frac{\gamma}{j}$. This implies
   \bes
   \sum_{j=1}^{i-1} \frac{a_j}{1 + a_j} = \sum_{j=1}^{i-1} \frac{\gamma}{\gamma + j} \leq \int_0^{i-1} \frac{\gamma}{\gamma+x}dx = \gamma \log \left( \frac{\gamma+i-1}{\gamma} \right).
   \ees
Here we have used that the continuous function $x \rightarrow \frac{\gamma}{\gamma+x}$ is decreasing on the interval $[0,i-1]$. We conclude that
   \bes
   z_i \leq 2 x_i \left( 1 + \gamma \log \left( \frac{\gamma + i - 1 }{\gamma} \right)\right) \leq 2 \|\bm{x}\|_\infty  \left( 1 + \gamma \log \left( \frac{\gamma + n - 1 }{\gamma} \right) \right).
   \ees
   This implies
   \bes
   \|\bm{z}\|_\infty \leq 2 \|\bm{x}\|_\infty \left( 1 + \gamma \log \left( \frac{\gamma + n - 1}{\gamma} \right) \right)
   \ees
   and the proof is complete.
 \end{proof}

By Lemma \ref{lem:exponential-growth} the eigenvectors grow at least exponentially when $\gamma = \frac{c}{b} = m$. However, if $b > 0$ and $c > 0$, then the eigenvectors are well-conditioned with respect to componentwise relative perturbations of the matrix because 
 \bes
 \kappa_\infty(\bm{B},\bm{f}) \leq 2 \left( 1 + m \log(2) \right).
 \ees
This is less surprising when we consider that the eigenvectors can be computed using additions and divisions of real numbers which are strictly positive.

\section{Upper triangular problems} \label{sec:upper-triangular-problems}

The choice of using lower rather than upper triangular matrices was made exclusively for pedagogical reasons. 
We find it simpler to apply the well-ordering principle when we are moving forward.
We pass from lower triangular eigenvalue problems $\bm{A}\bm{x}=\lambda \bm{x}$ to equivalent upper triangular problems with ease.
Simply replace the lower triangular matrices $\bm{A}$ and $\bm{X}$ with the upper triangular matrices $\bm{A}' = \bm{J}\bm{A}\bm{J}$ and $\bm{X}' = \bm{J}\bm{X}\bm{J}$ where $\bm{J}$ is the anti-diagonal identity matrix.
This similarity transformation reverses the numbering of the rows and columns of the matrices $\bm{A}$ and $\bm{X}$.
The specific example given in the introduction is transformed into
\be
\bm{A}' = 
\bbm
5 & -5 & -5 & -5 & -5 \\
& 4 & -5 & -5 & -5 \\
& & 3 & -5 & -5 \\
& & & 2 & -5 \\
& & & & 1 
\ebm, \quad \bm{X'} =
\bbm
1 & 5 & 15 & 35 & 70 \\
& 1 & 5 & 15 & 35 \\
& & 1 & 5 & 15  \\
& & & 1 & 5 \\
& & & & 1 &
\ebm.
\ee

\section{Conclusion} \label{sec:conclusion}

We have shown that there exists matrices for which the eigenvalues and eigenvectors are well-conditioned with respect to componentwise relative perturbations of the matrix. 
However, the eigenvectors cannot be computed using \emph{regular} substitution because they exceed the representational range.
These matrices can be used to test subroutines for computing eigenvectors as these subroutines must be able to deal successfully with overflow.

\bibliographystyle{acm}
\bibliography{references}

\begin{thebibliography}{1}

\bibitem{anderson1991robust-triangular}
{\sc Anderson, E.}
\newblock {LAPACK Working Note No. 36: Robust Triangular Solves for Use in
  Condition Estimation}.
\newblock Tech. rep., USA, 1991.

\bibitem{lapack}
{\sc Anderson, E., Bai, Z., Bischof, C., Blackford, S., Demmel, J., Dongarra,
  J., {Du Croz}, J., Greenbaum, A., Hammarling, S., McKenney, A., and Sorensen,
  D.}
\newblock {\em {LAPACK Users' Guide}}, third~ed.
\newblock Society for Industrial and Applied Mathematics, Philadelphia, PA,
  1999.

\bibitem{moon2016accelerating-eigenvector}
{\sc Moon, T., and Poulson, J.}
\newblock {Accelerating eigenvector and pseudospectra computation using blocked
  multi-shift triangular solves}.
\newblock {\em CoRR abs/1607.01477\/} (2016).

\bibitem{myllykoski2020introduction-to}
{\sc Myllykoski, M., and {Kjelgaard Mikkelsen}, C.~C.}
\newblock {Introduction to StarNEig---A Task-Based Library for Solving
  Nonsymmetric Eigenvalue Problems}.
\newblock In {\em Parallel Processing and Applied Mathematics\/} (Cham, 2020),
  R.~Wyrzykowski, E.~Deelman, J.~Dongarra, and K.~Karczewski, Eds., Springer
  International Publishing, pp.~70--81.

\end{thebibliography}

\end{document}